\documentclass[12pt]{article}
\usepackage{amsfonts}
\usepackage[pctex32]{graphics}
\usepackage[dvips]{graphicx}
\usepackage{amsthm}
\usepackage{amsmath}
\usepackage{amssymb}
\usepackage{fontenc}
\usepackage[all]{xy}
\usepackage{latexsym}
\usepackage{layout}
\usepackage{epsfig}
\usepackage{graphicx}
\usepackage{pstricks,pst-node,pst-plot} 
\usepackage{color} 
\title{Waiting Time Distribution for the Emergence of Superpatterns}
\author {Anant P.~Godbole and Martha Liendo\\
Department of Mathematics and Statistics\\
East Tennessee State University}
\begin{document}
\def\qed{\vbox{\hrule\hbox{\vrule\kern3pt\vbox{\kern6pt}\kern3pt\vrule}\hrule}}
\def\ms{\medskip}
\def\n{\noindent}
\def\ep{\varepsilon}
\def\G{\Gamma}
\def\lr{\left(}
\def\ls{\left[}
\def\rs{\right]}
\def\lf{\lfloor}
\def\rf{\rfloor}
\def\lg{{\rm lg}}
\def\lc{\left\{}
\def\rc{\right\}}
\def\rr{\right)}
\def\ph{\varphi}
\def\p{\mathbb P}
\def\nk{n \choose k}
\def\a{\cal A}
\def\s{\cal S}
\def\e{\mathbb E}
\def\v{\mathbb V}
\def\l{\lambda}
\newcommand{\bsno}{\bigskip\noindent}
\newcommand{\msno}{\medskip\noindent}
\newcommand{\oM}{M}
\newcommand{\omni}{\omega(k,a)}
\newtheorem{thm}{Theorem}[section]
\newtheorem{con}{Conjecture}[section]
\newtheorem{claim}[thm]{Claim}
\newtheorem{deff}[thm]{Definition}
\newtheorem{lem}[thm]{Lemma}
\newtheorem{cor}[thm]{Corollary}
\newtheorem{rem}[thm]{Remark}
\newtheorem{prp}[thm]{Proposition}
\newtheorem{ex}[thm]{Example}
\newtheorem{eq}[thm]{equation}
\newtheorem{que}{Problem}[section]
\newtheorem{ques}[thm]{Question}
\providecommand{\floor}[1]{\left\lfloor#1\right\rfloor}
\maketitle
\begin{abstract}
Consider a sequence $\{X_n\}_{n=1}^\infty$ of i.i.d.~uniform random variables taking values in the alphabet set $\{1,2,\ldots,d\}$.  A {\it k-superpattern} is a realization of $\{X_n\}_{n=1}^t$ that contains, as an embedded subsequence, each of the non-order-isomorphic subpatterns of length $k$.  We focus on the (non-trivial!) case of $d=k=3$ and study the waiting time distribution of $\tau=\inf\{t\ge7:\{X_n\}_{n=1}^t\ {\rm is\ a\  superpattern}\}$. 
\end{abstract}
\section{Introduction}    A string of integers with values from the set $\{1,2,\ldots,d\}$ (equivalently, a word on the $d$-letter alphabet) is said to contain a pattern if any \emph{order-isomorphic} subsequence of that pattern can be found within that word. For example, the word $5371473$ contains the subsequences $571$, $574$, and $473$, each of which is order-isomorphic to the string $231$. We call the string $231$ the pattern that is contained in the word since it is comprised of the lowest possible ordinal numbers that are order isomorphic to any of these three sequences.
In the literature, the term \emph{pattern} is often reserved for strings of characters in which each character is unique. This traditional definition of pattern is adhered to in this paper, while the term \emph{preferential arrangement} denotes those strings of characters in which repeated characters are allowed, but not necessary. 
The word $5371473$ in the previous example also contains the subsequences $373$ and $343$ which are both order-isomorphic to the string $121$; thus both the string $121$ and the string $231$ are preferential arrangements contained in the parent string.  This order isomorphism on the preferential arrangements is equivalent to a \emph{dense ranking system}, where items that are equal receive the same ranking number, and the next highest item(s) receive the next highest ranking number.   The number of preferential arrangements of length $n$ on $n$ symbols is given by the sequence of ordered Bell numbers, whose first few elements are $1, 3, 13, 75,\ldots$; see, e.g., \cite{sloane}.

The systematic study of pattern containment was first proposed by Herb Wilf in his 1992 address to the SIAM meeting on Discrete Mathematics. However, most results on pattern containment deal more directly with \emph{pattern avoidance}, specifically the enumeration and characterization of strings which avoid a given pattern or set of patterns.  The first results in this area are due to Knuth \cite{knuth}.  For example, if $\pi\in S_n$ is a random permutation (not word) then the probability that it avoids the pattern 123 is given by $\frac{C_n}{n!}$, where $C_n=\frac{{{2n}\choose {n}}}{n+1}$ are the Catalan numbers.  The number of 132, 231, 213, 312, and 321-avoiding permutations are also given by the Catalan numbers, which by Stirling's approximation are asymptotic to $K\cdot\frac{4^n}{n^{3/2}}$ for some constant $K$.  The Stanley-Wilf conjecture, namely that the number of permutations that avoid a fixed $k$-pattern is asymptotic to $C^n$ for some constant $0<C<\infty$, was proved in \cite{marcus}.

Of the few results available on pattern containment, most deal with specified sets of patterns contained in fixed length permutations, i.e. strings without repeated letters; here we cite the work of in \cite{albert}, \cite{barton}, \cite{burstein2}, \cite{eriksson}, \cite{miller}.
Research in this area mainly includes enumerating maximum occurrences of a given set of patterns (``packings"), which may only include one pattern, contained in a permutation of fixed length.   Burstein et al. \cite{burstein} have expanded this research further by not only removing the permutation requirement, thereby allowing for repeated letters in the word that is to contain the set of patterns, but also allowing repeated letters within the patterns themselves.  This work, and the references therein, seem to be closest in spirit to the work undertaken in the present paper.
We are specifically interested in the problem in \cite{burstein} regarding the word length required for a word to contain all preferential arrangements of a given length. We define a \emph{superpattern}, to be a word which contains all preferential arrangements of a given length.   Given $k,d\in{\mathbb Z}^+$, let $n(k,d)$ be the smallest string that contains all preferential arrangements of length $k$ on an alphabet of size $d$.  Since $n(k,d)=n(k,k)$ for $d\ge k$, it suffices to consider the case $d\le k$.  The authors of \cite{burstein} prove the following results:
\begin{lem}  $n(2,2)=3$ and for any $d\ge 3$, $n(d,d)\le d^2-2d+4$.
\end{lem}
\begin{lem}  For any $k\ge d\ge 3$, $n(k,d)\le(k-2)d+4$.
\end{lem}
They also conjecture that for all $d\ge 3$, $n(d,d)=d^2-2d+4$, which they argue is a very hard open problem.

In this paper, we tackle the following random version of the extremal work mentioned in the previous paragraph:  Consider a sequence $\{X_n\}_{n=1}^\infty$ of i.i.d.~uniform random variables taking on values in the alphabet set 

\noindent $\{1,2,\ldots,d\}$.  A {\it k-superpattern} is a realization of $\{X_n\}_{n=1}^t$ that contains, as an embedded subsequence, each of the preferential arrangements of length $k$.  After disposing off the case of $d=k=2$ in Section 2, we focus on the (non-trivial!) case of $d=k=3$ in Section 3, and study the waiting time distribution of $\tau=\inf\{t\ge7:\{X_n\}_{n=1}^t\ {\rm is\ a\  superpattern}\}$.   Here the infimum is taken over $t\ge 7$ in light of Lemmas 3.2 and 3.3 below.  As pointed out in Fu \cite{fu}, such problems are hard even for small $k$; there he studies the number of occurrences of the pattern 123 in a random permutation.  Another probability distribution that is in the spirit of the work undertaken here can be found in \cite{burton}, where the authors study the distribution of the first occurrence of a 3-ascending pattern.  It would be interesting, moreover, to see if the Markov chain embedding method (Fu and Koutras \cite{koutras}, Balakrishnan and Koutras \cite{bala}) can be used to good effect to make further progress in this area.

We end this section with some analogies drawn from \cite{omni}.  If, instead of considering preferential arangements, we ask for the waiting time $W$ until every {\it word} of length $k$ over a $d$-letter alphabet is seen, then the problem becomes both easier, in the sense that $\e(W)$ and $\v(W)$ can be easily computed, but elusive as far as the exact waiting time distribution is concerned.  It is shown in \cite{omni} that 
the distribution of $W$ is the same as that of the waiting time until $k$ disjoint coupon collections from the coupon set $\{1,2,\ldots,d\}$ are obtained. Further analyses and limit theorems are given in that paper.
\section{Binary Alphabet}    Some further classification of superpatterns is necessary for clarity in this paper. Let a \emph{minimal superpattern} be a superpattern in which no two adjacent letters are the same.  A \emph{minimum superpattern} is a minimal superpattern of the shortest length possible, i.e., one  in which every letter is necessary for the containment of all preferential arrangements. 
Let a \emph{strict superpattern} be a superpattern in which the last letter of the superpattern is needed to complete one of the preferential arrangements contained in the superpattern. Clearly, all minimum superpatterns are strict superpatterns, but not conversely.  Specifically, a strict superpattern may contain extraneous repeat letters; e.g., for $k=d=2$, 121 is a minimum superpattern, but 111221 is a strict non-minimum superpattern.

In the binary case, a superpattern is a word that contains all the preferential arrangements, namely $11$, $12$, and $21$.  The waiting time $\tau$ for a binary string to be a superpattern satisfies: $\tau=n$ iff there exist precisely two runs among the first $n-1$ letters of the word and the $n$th letter must be the letter that correctly completes a minimum superpattern.   The number of ways to partition $n-1$ letters into $2$ non-empty parts is $n-2$. Since there are a total of $2$ minimum superpatterns, namely $121$ and $212$, there are $2(n-2)$ words of length $n$ that satisfy the required conditions. Therefore the probability that a word on $n$ letters contains all preferential arrangements for $k=d=2$ is $$\p(\tau=n)=p_{(2,n)} = \frac{2(n-2)}{2^{n}}= \frac{n-2}{2^{n-1}}.$$

It follows that 
\begin{eqnarray} \e(\tau)&=& \sum_{n\geq 3} \frac{n(n-2)}{2^{n-1}}\nonumber\\
&=&\frac{1}{2}\sum_{n\ge 3}\frac{n(n-1)}{2^{n-2}} -\sum_{n\ge3}\frac{n}{2^{n-1}}\nonumber\\
&=&\frac{1}{2}\lr16-2\rr-(4-1-1)\nonumber\\&=&5,\end{eqnarray} 
in contrast to the fact that the waiting time for all words of length 2 to appear as subsequences is the waiting time for two disjoint coupon collections of two ``toys," which equals 3+3=6.
 Similarly, the variance is found to be 
\begin{eqnarray*} V(\tau) &=& \e(\tau^2)-[\e\tau)]^2\nonumber\\
&=& \frac{1}{4}\sum_{n\geq 3} \frac{n(n-1)(n-2)}{2^{n-3}} +5-25\nonumber\\
&=&\frac{1}{4}\cdot96+5-25\nonumber\\
&=&4,\end{eqnarray*} and the (rational) generating function is 
 \begin{eqnarray*} G_2(t) &=& \sum_{n\geq 3}\frac{t^n(n-2)}{2^{n-1}}\\ &=& \frac{t^3}{(2-t)^2}.\end{eqnarray*}

\section{Ternary Alphabet}  The sitation becomes vastly more complicated when $d=k=3$.  By way of comparison, we note that the expected waiting time for a single coupon collection, i.e., until one of each of the three letters of the alphabet is seen, is 1+1.5+3=5.5, so that the expected waiting time till each of the 27 ternary words is seen as a subsequence is $3\cdot5.5=16.5$.  How much less do we expect to have to wait till the string becomes a superpattern that contains each of the 13 preferential arrangements of three-letter words on a ternary alphabet,  namely $111$, $112$, $121$, $211$, $122$, $212$, $221$, $123$, $132$, $213$, $231$, $312$, and $321$, as subsequences?   Throughout the rest of the paper, we will refer to superpatterns in the context of this section as {\it superpatterns for} $[3]^3$, and denote the length of the superpattern by $n=n(3,3)$.  Following the notation of \cite{bona}, let $\pi=\pi_1, \pi_2, \ldots, \pi_k$ be a partition of $[n]$, and $\pi_i$ denotes a block of $\pi$. Then $a=(a_1, a_2, \ldots, a_k)$ is a partition of the integer $n$ where $a_i=\vert \pi_i \vert$ and $a_1 \geq a_2 \geq \cdots \geq a_k$. For example, if $n=7, k=3$, then one such partition of $7$ is $(5,1,1)$, and we will think of this as corresponding to the number of letters of the three types in the superpattern.     It should be noted that for any minimal superpattern no $a_i > \lceil \frac{n}{2} \rceil$, since this would cause adjacent letters to be the same. This fact combined with the following lemma prove very useful in determining the word length of superpatterns for $[3]^3$. 
\begin{lem}
Any superpattern for $[3]^3$ contains a $jk$ and a $kj$ pattern (as a subsequence) both before and after at least one $i$, where $i,j,k \in [3]$ with $i \neq j \neq k$.
\end{lem}
\begin{proof}  Let $\sigma$ be a superpattern for $[3]^3$ and let $i,j,k \in [3]$ with $i \neq j \neq k$.
Assume $\sigma$ does not contain a $jk$ pattern before an $i$. Then $\sigma$ does not contain the pattern $jki$ and $\sigma$ is not a superpattern for $[3]^3$. This is a contradiction and therefore $\sigma$ contains a $jk$ pattern before at least one $i$. 
The cases for $\sigma$ containing a $jk$ pattern after an $i$, $kj$ pattern before an $i$, and $kj$ pattern after an $i$ follow in a similar manner.
\hfill\end{proof}

It is clear, since ${5\choose 3}=10<13$, that there are no strict minimal superpatterns for $n=3$, $n=4$, or $n=5$.  Thus the smallest value of $n(3,3)$ is at least 6.
\begin{lem}
There are no strict minimal superpatterns of length $n=6$.\end{lem}
\begin{proof}The integer $6$ can be partitioned into $3$ parts in three ways, namely $(4,1,1)$, $(3,2,1)$, and $(2,2,2)$. 

Consider a strict minimal superpattern with $(a_1,a_2,a_3)=(4,1,1)$. Then there exists an $a_i> \lceil \frac{n}{2} \rceil=3$, causing two adjacent letters to be the same letter, which contradicts the fact that $\sigma$ is a strict minimal superpattern. 
Next, consider a strict minimal superpattern with $(a_1,a_2,a_3)=(3,2,1)$, so that $a_3=1$. Let $i$, the singleton letter, be the $r$th letter of the six letter string. Then $r\geq 4$ since there exists both a $jk$ and a $kj$ pattern before $i$ and $r\leq 3$ since there exists both a $jk$ and a $kj$ pattern after $i$. Thus no such $r$ exists and therefore there is no strict minimal superpattern with 3, 2, and 1 letters of the three types.
Finally, consider a strict minimal superpattern with $(a_1,a_2,a_3)=(2,2,2)$. Then there does not exist an $a_i\geq 3$ and thus no $111$ pattern exists, which contradicts the fact that we have a strict minimal superpattern. \hfill\end{proof}
\begin{lem}
There exist seven strict minimal superpatterns of length $n=7$ up to isomorphism.
\end{lem}
\begin{proof}The integer $7$ can be partitioned into $3$ parts in four ways, namely $(5,1,1)$, $(4,2,1)$, $(3,3,1)$, and $(3,2,2)$.

Case 1: Consider a strict minimal superpattern corresponding to a $(5,1,1)$  partition. Then there exists an $a_i> \lceil \frac{n}{2} \rceil=4$, causing two adjacent letters to be the same, which contradicts the strict minimality of the  superpattern. This case is thus vacuous.

Case 2: Consider a strict minimal superpattern with partition structure $(4,2,1)$ with $a_i=1$, $a_j=4$, and $a_k=2$. Let the $r$th letter of the string equal $i$ for some $r \in [7]$. Then $r\geq 4$ since there exists both a $jk$ and a $kj$ pattern before $i$ and $r\leq 4$ since there exists both a $jk$ and a $kj$ pattern after $i$. Therefore $r=4$. Since there are four instances of the letter $j$, and two of the letter $k$, we see that the first three letters of the string must correspond to the last three letters of the string.  Therefore (up to isomorphism) there exists one such strict minimal superpattern having 4, 2, and 1 occurrences of the three letters.  Denote this superpattern by 1213121.

Case 3: Consider a strict minimal superpattern with 3, 3, and 1 occurrences of the letters.  
Set $a_i = 1, a_j = 3, a_k = 3$. Let the $r$th letter of the string be the singleton $i$.  Then $r\ge4$ since there
exists both a $jk$ and a $kj$ pattern before $i$, and  $r\le4$ since there exists both a $jk$ and
a $kj$ pattern after $i$. Thus $r=4$. Since there are 3 instances of each of the letters $j$ and $k$, we see that the first and last three letters of the string must be comprised of $jkj$ and $kjk$ respectively.  Up to isomorphism, therefore, exists just one such strict minimal superpattern with partition structure $(3, 3, 1)$; we denote it by 1213212.

Case 4:  The case with partition structure $(3,2,2)$ is the most complicated case with five non-isomorphic solutions.  Consider a strict minimal superpattern 
with $a_i = 3, a_j = 2, a_k = 2$.  We focus on the most frequent letter. Let the $r$th, $s$th and $t$th letters be of the string be $i$ for some 
$i$, with $1\le r< s< t\le 7$. Since no two adjacent letters are the same letter, $3\le s\le 5$.

If $s = 3$, then $r = 1$ and $t = 5, 6,$ or $7$, since no two adjacent letters are the
same letter.  If $t = 5$, then there does not exist both a $jk$ and a $kj$ pattern before at
least one $i$, which contradicts Lemma 3.1. Therefore $t \ne 5$.  If $t = 6$, then we find that Lemma 3.1 is violated no matter in which of the six possible ways the two 2's and two 3's are arranged.  Thus $t\ne 6$.  If $t = 7$, then once again we see
and there does not exist a configuration of the other four letters for which Lemma 3.1 is satisfied.  Thus $t\ne 7$.

If $s = 4$, then $r = 1$ or 2 and $t = 6$ or 7 since no two adjacent letters are
the same. If $r=1, t=6$, the only feasible pattern is $ijkijik$.  If $r=1, t=7$, there are two solutions, namely $ijkijki$ and $ijkikji$.  If $r=2, t=6$, the single solution is $jikijik$, and, finally, if $r=2, t=7$, the single solution is $jikijki$.

It can be shown that no additional solutions exist for $s=5$.  This completes the proof.
\hfill\end{proof}
\begin{cor}
The length of a minimum superpattern for $[3]^3$ is $n(3,3)=7$.
\end{cor}
Burstein et al. (\cite{burstein}) give a constructive proof for $n(l, l) \le l^2 - 2l + 4$ and
conjecture that $n(l, l) = l^2 - 2l + 4$. The corollary above characterizes the solutions for the case $l=3$.  
The seven unique strict minimal superpatterns of length $n=7$, up to isomorphism, 
are 1213121, 1213212, 1231213, 1231231, 1231321, 1232123, and 1232132. Since there are 3! ways to permute the letters isomorphically in each
strict minimal superpattern of length $n = 7$, we obtain a total of 3!(7) = 42 strict
minimal superpatterns of length $n = 7$.  {\it These are also the minimum superpatterns}.

Next, we consider the total number of minimal
superpatterns, up to isomorphism, for any any length $n\ge 8$. Since all minimal superpatterns are
comprised of an alternating pattern, then, up to isomorphism, the first two letters
can be fixed as $i$ and $j$ for $i, j \in [3]$ with $i \ne j$. There exist $2^{n-2}$ total words on
the remaining $n - 2$ positions that have alternating patterns since each letter can be
chosen in two ways. However, not all of these $2^{n-2}$ words will result
in a $[3]^3$-superpattern of  length $n$. The following lemma aids in determining the
number of candidate $n$-strings which fail to create a superpattern of $[3]^3$; this number, up to isomorphism, ends up being $(n-2)^2$.
\begin{lem}
Any strict minimal $n$-superpattern for $[3]^3$ contains
a minimum superpattern for $[3]^3$ with the last letter of the minimum superpattern
occurring on the last letter of the superpattern.
\end{lem}
\begin{proof}
Consider, up to isomorphism,  a strict minimal superpattern $\sigma$ of length $n$ 
for $[3]^3$. Let $\{i, j, k\} = [3]$. Without loss of
generality, let $\sigma(n) = i$ and $\sigma(n - 1) = k$. Then there exists some $\sigma(b_1) = i$ as the first occurrence of $i$ in $\sigma$, and, without
loss of generality, there exists $(\sigma(c_1), \sigma(c_2)) = (k,j)$ with $\sigma(c_1) = k$ as the first occurrence
of $k$ in $\sigma$, and $\sigma(c_2) = j$ as the last occurrence
of $j$ in $\sigma$ where $b_1  < c_2 < n-1$ since there exists both a $jk$ and a $kj$ pattern after
at least one $i$. If $b_1 > 3$ then there exists a $jk$ and a $kj$ pattern before it, causing $\sigma$ to
contain either a $jkjikjk$ or a $kjkikjk$ pattern, both of which are strict superpatterns
of length $n = 7$ and therefore $\sigma(n) = i$ is unnecessary for the containment of all
preferential arrangements. This contradicts the given fact that $\sigma$ is a strict minimal
superpattern. Therefore $b_1 \le 3$.

Case 1: If $b_1 = 3$, then $(\sigma(1), \sigma(2)) = jk$ or $kj$. If $(\sigma(1), \sigma(2)) = jk$, then $\sigma$ contains the minimum superpattern $jkikjki$ with the last letter of the minimum superpattern occurring on the last letter of $\sigma$. If $(\sigma(1), \sigma(2)) = kj$, then $
\sigma$  contains the minimum superpattern $kjikjki$ with the last letter of the minimum superpattern occurring on the last letter of $\sigma$.

Case 2: If $b_1 =2$,then $\sigma(1) = j$ or $k$.   If $\sigma(1)=j,$ then there exists the pattern $ki$ before $\sigma(c_2) = j$ since there exists a $ki$ pattern before at least one $j$ and thus it must also exist before the last $j$. Then $\sigma$ contains the minimum superpattern $jikijki$ 
with the last letter of the minimum superpattern occurring on the last letter of $\sigma$. If $\sigma(1)=k$ (here $c_1=1$), then there exists a $ji$ pattern before $\sigma(n-1) = k$ since there exists a $ji$ pattern before at least one $k$ and $\sigma(n-1)$ is the last occurrence of $k$. Since no two adjacent letters are the same letter, $\sigma(3)=j$ or $k$. If $\sigma(3)=j$, then (noting that there must be a $k$ between the third spot and the $c_2$th) $\sigma$ contains either a $kijikji$ on the first $n$ letters, or a $kijkijk$ or  $kijkjik$ pattern on the first $n-1$ letters.  In the first case, $
\sigma$  contains a minimum superpattern $kjikjki$ with the last letter of the minimum superpattern occurring on the last letter of $\sigma$.  In the second and third case, we find embedded minimum superpatterns on $n-1$ letters, and therefore $\sigma(n)=i$  is unnecessary for the the containment of all preferential arrangements. This contradicts the given fact that $\sigma$ is a strict minimal superpattern. If $\sigma(3)=k$, then $\sigma$ contains the minimum superpattern $kikjiki$ with the last letter of the minimum superpattern occurring on the last letter of $\sigma$. This is because there must be an $ik$ and a $ki$ after some $j$.

Case 3: If $b_1 = 1$, then $\sigma(2) = j$ or $k$. If $\sigma(2) = j$, then there exists a $ki$ pattern before $\sigma(c_2) = j$. Therefore $\sigma$ contains the minimum superpattern $ijkijki$ with the last letter of the minimum superpattern occurring on the last letter of the string. If $\sigma(2) = k$, then $\sigma(3) = i$ or $j$. If $\sigma(3) = i$, note that there exists a $ji$ pattern before $\sigma(n-1) = k$.  Thus $\sigma$ contains the minimum superpattern $ikijiki$ with the last letter of the minimum superpattern occurring on the last letter of the string. If $\sigma(3) = j$, note that there exists a $ki$ pattern (where $\sigma(2) = k$ is the $k$ of the pattern) before $\sigma(c_2) = j$ since there exists a $ki$ pattern before at least one $j$, and $\sigma(c_2)$ is the last occurrence of $j$. Thus the string contains the minimum superpattern $ ikjijki$ with the last letter of the minimum superpattern occurring on the last letter of $\sigma$.

Since any $i, j, k \in [3]$ can be permuted by isomorphisms, all strict minimal $n$-superpatterns for $[3]^3$; $n\ge8$, contain a minimum superpattern with the last letter of the minimum superpattern occurring on the last letter of the string.  This completes the proof.
\hfill\end{proof}

It now follows that the strict minimal strings that fail to create a superpattern of $[3]^3$ do not contain a complete embedding of one of the strict minimal superpatterns of length seven (again, for $n=7$ these are the same as the minimum superpatterns), since by Lemma 3.5 all strict minimal superpatterns contain a strict minimal superpattern of length seven. All the words contain some portion of a strict minimal superpattern of length seven up to isomorphism, since the first two letters are fixed as $i$ and $j$ and each strict minimal superpattern of length seven can be written in the same manner. Let an ``$i$-fold progression" count the number of the $2^{n-2}$ words which begin with $ij$ and contain the first through the $i$th letters of a unique strict minimal superpattern of length seven, but not the $i+1$st letter. Then 2-fold progression is guaranteed by the fixed $i$ and $j$ occurring on the first and second positions of each word. The third position must be an $i$ or a $k$ since no two adjacent letters are the same letter. Let the strict minimal superpatterns of length seven with the first three positions containing the pattern $iji$ be called type A patterns, with the strict minimal superpatterns of length seven with the first three positions containing the pattern $ijk$ being called type B patterns.

 First, consider the strict minimal superpatterns of type A, namely $ijikiji$ and $ijikjij$, where $i,j,k \in [3]$ with $i \neq j \neq k$. A word that satisfies $3$-fold progression contains the pattern $iji$ on the first three positions, but no $k$ afterwards. There is one such word, namely $ijijij \ldots$, which satisfies a $3$-fold progression. 

For a $4$-fold progression to occur, the word must contain the pattern $iji$ on the first three positions followed by a $k$ which has no $i$ or $j$ after it, otherwise a $5$-fold progression will occur. There is only one such word, namely $ijijij \ldots k$, where the only occurrence of $k$ is at the end of the word. 

There are $2(n-4)$ Type A words that exhibit a 5-fold progression, namely any word which follows the pattern $ijijij \ldots kikiki \ldots $ and $ijijij \ldots kjkjkj \ldots$, where the $k$ can be inserted in any position other than the first, second, third, or $n$th. 

In order for a word to contain a $6$-fold progression, it must contain the $5$-fold progression $ijijij \ldots kikiki \ldots$ followed by a $j$ or $ijijij \ldots kjkjkj \ldots$ pattern followed by an $i$.
 This corresponds to all the ways in which two non-consecutive choices can be made from $n-3$ spots for the $k$ and the sixth letter of the progression, so there are $2 {n-4 \choose 2}$ such words, namely $ijijij \ldots kikiki \ldots jkjkjk \ldots$ 

\noindent and $ijijij \ldots kjkjkj \ldots ikikik \ldots$. 

Therefore the total count for the number                                                                                                           
of words which do not contain a complete embedding of one of the type A strict minimal superpatterns of length seven is \begin{eqnarray*} \beta_A(n) &=& 1 + 1 + 2(n-4) + 2{n-4 \choose 2} \\ &=& n^2 - 7n + 14. \end{eqnarray*}

 Next, consider the strict minimal superpatterns of type B, namely $ijkijki$, $ijkikji$, $ijkijik$, $ijkjijk$ and $ijkjikj$, where $i,j,k \in [3]$ with $i \neq j \neq k$. There exist no words that satisfy a $3$-fold progression since all words containing the pattern $ijk$ on the first three positions contain either an $i$ or a $j$ immediately afterwards and there exists either the pattern $ijki$ or the pattern $ijkj$ on at least one of the strict minimal superpatterns of type B, causing at least a $4$-fold progression to occur. 

For a $4$-fold progression to occur, the word must contain either the pattern $ijki$ on the first four positions with no $j$ or $k$ afterwards, which is impossible, or the pattern $ijkj$ on the first four positions with no $i$ afterwards, otherwise a $5$-fold progression will occur. There is only one such word, namely $ijkjkjk \ldots$.

 For a $5$-fold progression to occur using the pattern $ijki$ as a basis pattern on the first four positions, the word must contain either the pattern $ijkij$ on the first five position with no $i$ or $k$ afterwards, which is impossible, or the pattern $ijkik$ on the first five positions with no $j$ afterwards, otherwise a $6$ fold progression will occur. There is only one such word, namely $ijkikiki \ldots$. 
For a $5$-fold progression to occur using the pattern $ijkj$ as a basis pattern on the first four positions, there is only one possibility, namely $ijkjkjk \ldots i$, where the only occurrence of $i$ after position four is at the end of the word.  Since any other occurrence of $i$ on the $(n-5)$ remaining positions (other than the last position) results in a $6$-fold progression, there are $n-5$ ways for the word to contain a $6$-fold progression for each possible letter that can follow $i$ using the pattern $ijkj$ as a basis pattern on the first four positions. A $6$-fold progression is contained in the word if the pattern $ijkjij$ is not followed by a $k$ or the pattern $ijkjik$ is not followed by a $j$. 
There are $2(n-5)$ such words.  A word can also contain a $6$-fold progression using the pattern $ijkij$ as a basis pattern on the first five positions if the word contains either the pattern $ijkiji$ on the first six positions with no $k$ afterwards or the pattern $ijkijk$ on the first six positions with no $i$ afterwards.
 There exists only one such word for each of these cases, namely $ijkijijij \ldots$ and $ijkijkjkjk \ldots$. Lastly, a word can also contain a $6$-fold progression if it contains the pattern $ijkik$ on the first five positions followed by a $j$ on one of the $n-5$ remaining positions that is not followed by an $i$. 
There are $n-5$ such words, namely any word that follows the pattern $ijkikiki \ldots jkjkjk \ldots$. Therefore the total count for the number of words which do not contain a complete embedding of one of the type B strict minimal superpatterns of length seven is \begin{eqnarray*} \beta_B(n) &=& 1 + 1 + 1 + 2(n-5) + 1 + 1 + (n-5) \\ &=& 3n - 10. \end{eqnarray*}

 Therefore the total number of words that do not contain a complete embedding of one of the strict minimal superpatterns of length seven and thus fail to create a superpattern of $[3]^3$ is \begin{eqnarray*} \beta_{total}(n) &=& n^2 - 7n + 14 + 3n - 10 \\ &=& (n-2)^2, \end{eqnarray*}
making the total number of minimal superpatterns of any length $n\ge 7$, up to isomorphism, equal to \begin{eqnarray*} \Gamma_{total}(n) &=& 2^{n-2} - (n-2)^2. \end{eqnarray*} The sequence generated by $\Gamma_{total}(n)$ existed previously in \cite{sloane} as entry number $A024012$, but with little context.  We have now added the ``superpattern origin" of the sequence to that OEIS entry.
\begin{lem}
For all $n \geq 7$, the total number $S_\mu(n)$ of strict minimal superpattern of length $n$ is given by $S_\mu(n)=(n-4)^2 - 2.$
\end{lem}
\begin{proof}Up to isomorphism, the number of strict minimal superpattern of length $n$ will equal the total number of minimal superpatterns of length $n$ minus any non-strict minimal superpatterns of length $n$. 
The total number of non-strict superpatterns of length $n$ is equal to the total number of minimal superpatterns of length $n-1$ times $2$, since the last letter is unnecessary in a non-strict superpattern for the completion of any preferential arrangement of $[3]^3$, making the word on the first $n-1$ letters a valid minimal superpattern of length $n-1$ and there are $2$ choices for the $n$th letter since no two adjacent letters in the word are the same letter. Therefore, 

\begin{eqnarray*} S_\mu(n) &=& [2^{n-2}-(n-2)^2]-2[2^{n-3}-(n-3)^2]\\ &=& (n-4)^2 - 2, \end{eqnarray*} 
as asserted.  The sequence generated by $S_\mu(n)$ existed as entry number $A008865$ in \cite{sloane}, but with little context.  We have added the above origin.\hfill\end{proof}

\begin{lem}
The number $S_a(n)$ of strict $n$-superpatterns in which there exist possible occurrences of adjacent repeated letters is given by $S_a(n)=\sum_{m=7}^{n}[(m-4)^2-2]{n-2 \choose m-2}$.
\end{lem}
\begin{proof}Any strict superpattern of length $n$ in which there exist occurrences of two adjacent and repeated letters will contain an embedded occurrence of a strict minimal superpattern of length $m$, where $7 \leq m \le n$. Therefore all such superpatterns are found by inserting $n-m$ letters which cause two adjacent letters to be the same into strict minimal superpatterns of length $m$.

 These insertions can take place anywhere in the word except before the last letter since an occurrence of two adjacent letters as the same letter at the end of the word contradicts the strictness of the superpattern. 
Therefore there are $n-m$ insertions of identical ``balls" into $m-1$ possible positions and there are ${{m-1+(n-m)-1} \choose {m-2}}={{n-2}\choose{m-2}}$ ways to do this.   Since this insertion of the appropriate number of repeats can be done for all strict minimal superpatterns of length $m$, $7\le m\le n$, 
\begin{eqnarray*} S_a(n) &=&\sum_{m=7}^{n}S_\mu(m){n-2 \choose m-2}\\ &=&\sum_{m=7}^{n}[(m-4)^2-2]{n-2 \choose m-2}, \end{eqnarray*} which finishes the proof.\hfill\end{proof}  We now state the main result of this paper:
\begin{thm}
For all $n\geq 7$ the total number of strict superpatterns of length $n$ is given by $S(n)=6\sum_{m=7}^{n}[(m-4)^2-2]{n-2 \choose m-2},$ and thus 
the probability distribution of the waiting time $\tau$ for all preferential arrangements of $[3]^3$ to occur as a subsequence is 
$$\p(\tau=n)=p_{(3,n)} = \frac{6}{3^{n}}\sum_{m=7}^{n}[(m-4)^2-2]{n-2 \choose m-2}.$$
\end{thm}
\begin{proof}The first part of the proof follows immediately from Lemma 3.7 and the fact that there are 6 isomorphic arrangements for any superpattern.  The second part follows due to the immediate correspondence between a strict superpattern and the waiting time, and the fact that each of the $3^n$ sequences are equally likely.
This completes the proof.\hfill\end{proof}
Computation of moments is now routine.  We have
\begin{eqnarray*}{}&&\e(\tau)\\ \enspace&&= \sum_{n=7}^{\infty} \frac{6n}{3^{n}}\sum_{m=7}^{n}[(m-4)^2-2]{n-2 \choose m-2}\\ \enspace&&= 6 \sum_{m=7}^{\infty}( m^2 - 8m +14) \sum_{n=m}^{\infty}\frac{n{n-2 \choose m-2}}{3^{n}}\\ \enspace \enspace&&= 6 \sum_{m=7}^{\infty} (m^2 - 8m +14) \sum_{n=m}^{\infty}\frac{{n-2 \choose m-2}}{3^{n}}\\ \enspace&&+ 6 \sum_{m=7}^{\infty} (m^2 - 8m +14) \sum_{n=m}^{\infty}\frac{(n-1){n-2 \choose m-2}}{3^{n}}\end{eqnarray*}
\begin{eqnarray*} \enspace&&= \sum_{m=7}^{\infty}p_{(3,n)} + 6 \sum_{m=7}^{\infty} \frac{(m^2 - 8m +14)(m-1)}{3^m} \sum_{l=m-1}^{\infty}\frac{{l \choose m-1}}{3^{l-(m-1)}}\\ \enspace&&= 1 + 6 \sum_{m=7}^{\infty} \frac{(m^3 - 9m^2 + 22m -14)}{2^m} \\  \enspace&&= 13.5625,\end{eqnarray*}
and similar computations, not shown in detail, yield the generating function $G_3(t)$:

\begin{eqnarray*}
G_3(t) &=& \sum_{n=7}^{\infty} \frac{6t^n}{3^{n}}\sum_{m=7}^{n}[(m-4)^2-2]{n-2 \choose m-2}\\ &=& \frac{2t^7(16t^2-63t+63)}{(3-t)^5(3-2t)^3}.\end{eqnarray*}

\section{Open Questions}  The key questions we would like to see resolved are as follows:  (i) Can other methods, particularly generation function techniques \cite {wilf2} or the Markov chain embedding technique \cite {fu}, \cite{koutras} be used to give alternative proofs of our results and lead to generalizations for alphabets of size higher than 3?  One major complication to note is that a minimum superpattern for $[4]^4$ of length $12$ can be constructed using the construction method found in work by Burstein et al., but there exist strict superpatterns for $[4]^4$ of lengths larger than $12$ which do not contain one of the minimum superpatterns. One such example can be constructed using two copies of type A strict superpatterns for $[3]^3$ separated by a 4, i.e., $121312141213121$.  (ii) For $d\ge3$, can we obtain the exact distribution, in a not-too-complicated form, for the waiting time till all the {\it words} of length $k$ are obtained as subsequences?  NOTE:  This would be the waiting time for the completion of $k$ disjoint non-overlapping renewals of coupon collections with $d$ tokens; see \cite{omni}.

\section{Acknowledgments} The research of both AG and ML was supported by NSF Grant 0742364.  AG was further supported by NSF Grant 1004624.

\end{document}